\documentclass[12pt]{article}
\usepackage{amsmath, amsfonts, amsthm, amssymb, verbatim}
\usepackage{graphicx}
\usepackage[mathcal]{euscript}
\usepackage{amstext}
\usepackage{amssymb,mathrsfs}
\usepackage[latin1]{inputenc}
\newcommand{\R}{\mathbb R}
 
 \newcommand{\N}{\mathbb N}

\newcommand{\ve}{\varepsilon}
\newcommand{\vp}{\varphi}

\newcommand \loc    {\text{loc}}

\newcommand{\dive}{{\rm div}}

\newtheorem{theorem}{Theorem}[section]
\newtheorem{proposition}[theorem]{Proposition}
 \newtheorem{remark}[theorem]{Remark}
\newtheorem{lemma}[theorem]{Lemma}

\newtheorem{definition}[theorem]{Definition}

\begin{document}
\title{  Wellposedness for stochastic continuity equations 
with Ladyzhenskaya-Prodi-Serrin condition }

\author{Wladimir Neves$^1$, Christian Olivera$^2$}

\date{}

\maketitle

\footnotetext[1]{Departamento de Matem\'atica, Universidade Federal
do Rio de Janeiro, C.P. 68530, Cidade Universit\'aria 21945-970,
Rio de Janeiro, Brazil. E-mail: {\sl wladimir@im.ufrj.br.}}

\footnotetext[2]{Departamento de
 Matem\'{a}tica, Universidade Estadual de Campinas, 13.081-970, Campinas, SP, 
 Brazil. e-mail:  {\sl colivera@ime.unicamp.br.}

}


%
\begin{abstract}
We consider the stochastic divergence-free continuity equations with Ladyzhenskaya-Prodi-Serrin condition. 
Wellposedness is proved meanwhile uniqueness may fail for the deterministic PDE. 
The main issue of strong  uniqueness, in the probabilistic sense,  relies on stochastic characteristic method and 
 the generalized It\^o-Wentzell-Kunita formula. The stability property for the unique solution
is proved with respect to the initial data. Moreover, a persistence result is established 
by a representation formula.
 \end{abstract}
%
\maketitle

%

\section {Introduction} \label{Intro}

In this paper we establish wellposedness for stochastic divergence-free continuity equations. Namely, we consider 
the following Cauchy problem: Given an initial-data $u_0$,
find $u(t,x;\omega) \in \R$, satisfying 
\begin{equation}\label{trasport}
 \left \{
\begin{aligned}
    &\partial_t u(t, x;\omega) + \dive \Big(\, u(t, x;\omega)  \, \big(b(t, x) + \frac{d B_{t}}{dt}(\omega)\big ) \Big)= 0,
    \\[5pt]
    &u|_{t=0}=  u_{0},
\end{aligned}
\right .
\end{equation}
$\big( (t,x) \in U_T, \omega \in \Omega \big)$, where $U_T= [0,T] \times \R^d$, for
 $T>0$ be any fixed real number, $(d \in \N)$, $b:\R_+ \times
\R^d \to \R^d$ is a given vector field, 
with $\dive \,  b(t,x)= 0$, 
$B_{t} = (B_{t}^{1},...,B _{t}^{d} )$ is a
standard Brownian motion in $\mathbb{R}^{d}$ and the stochastic
integration is taken (unless otherwise mentioned) in the Stratonovich sense. 
In fact, through of this paper, we fix a stochastic basis with a
$d$-dimensional Brownian motion $\big( \Omega, \mathcal{F}, \{
\mathcal{F}_t: t \in [0,T] \}, \mathbb{P}, (B_{t}) \big)$.

\medskip
The Cauchy problem for the stochastic transport equation
$$
   \partial_t u(t, x;\omega) +   \big(b(t, x) + \frac{d B_{t}}{dt}(\omega)\big ) \cdot \nabla u(t, x;\omega)= 0
$$
has taken great attention recently, see for instance \cite{Falnanas},   \cite{CO4}, \cite{FGP2},  
\cite{Ku}, \cite{Ku3}, \cite{BrisLion}, and more recently the initial-boundary value problem in \cite{WNCO1}. 
Concerning the deterministic case of the problem \eqref{trasport}, also in a non-regular framework, the reader is mostly addressed to  \cite{DL} and  \cite{ambrisio}. Those papers
deal respectively with the Sobolev and the BV spatial regularity case, where the uniqueness proof  relies on commutators, see DeLellis \cite{lellis} for a 
nice review on that. The reader is directed towards the following references at the cited papers.

\medskip
The main issue in this paper is to prove uniqueness of weak  $L^{\infty}-$solution (see Definition \ref{defisolu})
of the Cauchy problem  (\ref{trasport}) for  vector fields  
\begin{equation}
\label{LPSC}
    \begin{aligned}
     &b\in L^{q}([0,T], (L^{p}(\mathbb{R}^{d}))^d), \quad p,q < \infty,
     \\[5pt]
    &p \geq 2, \quad q > 2, \quad \text{and} \quad \frac{d}{p} + \frac{2}{q} < 1.
    \end{aligned}
\end{equation}
The last condition \eqref{LPSC} is known 
in the fluid dynamic's literature 
as the Ladyzhenskaya-Prodi-Serrin condition,  with $\leq$ in place of $<$.    
Here, we do not assume  any differentiability (one of the main assumptions in \cite{Falnanas}), nor boundedness (also important in  
\cite{FGP2}) of the vector field  $b$. The uniqueness result, see Theorem \ref{uni}, is established 
using the transportation property of the continuity equation for divergence free 
vector fields.  Indeed, under the Ladyzhenskaya-Prodi-Serrin  condition for $b$, we are allowed to compose 
the solution $u$ to the transport equation with the
stochastic flow, in fact its inverse (see \eqref{itoassBac}), then 
we bring on it with all its space derivatives on the test function. Thus, avoiding the commutator
and the problems there in.  
Consequntly, we have sharpened the answer of the following question: 
Why noise improves the deterministic theory for 
transport/continuity equations?

\bigskip
In fact, that noise could improve the theory of transport equations was first discovered
by \cite{FGP2}. More precisely, the condition assumed in  \cite{FGP2} is  H\"{o}lder continuity and boundedness of $b$, and an integrability condition on the
divergence. Others results appear in  \cite{Falnanas} where no $L^{\infty}-$control on the divergence is required, however,   weak differentiability
is assumed.  Our result is more advanced in the sense  that we work with integrable and not   differentiability coefficients. Further, no boundedness of $b$ is
assumed. 
%

\bigskip
We recall that the Ladyzhenskaya-Prodi-Serrin condition \eqref{LPSC} (with local integrability, and $p,q \leq \infty$)
was first considered by  Krylov, R\"{o}ckner \cite{Krylov}. In that paper, they 
proved the existence and uniqueness
of strong solutions for SDE 
\begin{equation}\label{itoass}
X_{s,t}(x)= x + \int_{s}^{t}   b(r, X_{s,r}(x)) \ dr  +  B_{t}-B_{s},
\end{equation}
where  given $t \in [0,T ]$  and  $x \in \mathbb{R}^{d}$, it was shown that 
$$
    \mathbb{P} \big\{ \int_0^T |b(t,X_t)|^2 \ dt< \infty \big\}= 1.
$$
More recently, Fedrizzi, Flandoli see \cite{Fre1,Fre2} proved  
the  $\alpha$-H\"{o}lder  continuity of the  stochastic flow  $x \rightarrow X_{s,t} $
 for any $\alpha \in  (0, 1) $. Moreover,  they prove that it is a  stochastic flow of  homeomorphism.

\bigskip
Similarly, we may consider for convenience the inverse  $Y_{s,t}:=X_{s,t}^{-1}$, which satisfies the following backward stochastic
differential equations,
\begin{equation}\label{itoassBac}
Y_{s,t}(y)= y - \int_{s}^{t}   b(r, Y_{r,t}(y)) \ dr  - (B_{t}-B_{s}),
\end{equation}
for $0\leq s\leq t$. Usually $Y$ is called the time
reversed process of $X$.

\bigskip    
One of the main motivations to consider
problem \eqref{trasport} comes from the study of 
stochastic partial differential equations in fluid dynamics. In this 
direction, our ansatz is based on rational Continuum Mechanics.
Let $\psi$ be a physical 
quantity, which is a tensor field of order $m$. Also, we consider the supply and the flux of $\psi$, denoted 
respectively $\sigma_\psi$, $\phi_\psi$, which are  tensor fields of order 
$m$ and $m+1$. Then, the general stochastic balance equation has (at least formally) the form
$$
   \partial_t \psi + \dive \big( \psi \otimes \frac{d X_t}{dt} - \phi_\psi \big)= \sigma_\psi,
$$
with $X_t$ given for instance by 
$$
    X_{t}(x)= x + \int_{0}^{t}    \mathbf{v}(s, X_{s}(x)) \ ds  +  V_{t},
$$
where $\mathbf{v}$ is the velocity field, and $V_t$ is a stochastic process, 
which is not due necessarily to a Brownian motion, but posses for instance a Markov property. 
Therefore, our main assumption is to randomly perturb the motion of the physical quantity $\psi$. 
In particular, taking $\psi= \rho$, and $\psi= \rho \, \mathbf{v}$, 
the Cauchy problem for the  incompressible non-homogeneous stochastic Navier-Stokes equations may be written as
\begin{equation}\label{SNSEQ}
 \left \{
\begin{aligned}
    &\partial_t \rho + \dive \Big(\rho \, \big(\mathbf{v} + \frac{d V_{t}}{dt}\big )  \Big)= 0,
    \\
    &\dive \,  \mathbf{v}= 0,
    \\
    &\partial_t (\rho \, \mathbf{v}) + \dive \Big(\rho \, \mathbf{v} \otimes \big(\mathbf{v} + \frac{d V_{t}}{dt} \big) -T  \Big)= \rho \,f,
\\[5pt]
    &\rho(0)= \rho_{0}, \quad \mathbf{v}(0)= \mathbf{v}_0,
\end{aligned}
\right .
\end{equation}
where $\rho$ is the density, and $T$ is the stress tensor field given by
$$
      T= 2 \, \mu(\rho) \, D(\mathbf{v}) - p \, I_d   
$$
with the scalar function $p$ called pressure. Moreover, $D(\mathbf{v})$ is the symmetric part of the gradient of the velocity field, $\mu$ is the dynamic
viscosity, and $f$ is an external body force.
The above problem \eqref{SNSEQ} seems to us an onset of turbulence, which is a 
challenging phenomenon to understand in the (incompressible) fluid dynamics theory. 
The reader is further addressed to Flandoli \cite{FFNOISE},  Mikulevicius, Rozovskii \cite{MR},
and references therein. 

\bigskip
Moreover, the uniqueness result 
obtained by the authors for the stochastic continuity equation here in this paper, have to open 
new directions to establish existence of solutions to stochastic scalar conservation laws, 
for non-homogenous 
flux functions $f(t,x,u)$ with low regularity in the time-space variables $(t,x)$. In this direction, we recall the stochastic averaging  lemmas for
kinetic equations studied recently by Lions, Perthame, Souganidis \cite{LPS}.

\bigskip
The plan of exposition is as follows: In the rest of this section, we  shall prove existence of
weak  $L^{\infty}-$solutions via the  It\^o-Wentzell-Kunita formula. In Section 2, 
we prove the uniqueness of weak $L^{\infty}-$solutions. Moreover, we show that the unique solution is given by a representation formula, in terms of the initial data and the stochastic flow
associated to equation (\ref{trasport}). In Section 3, we present stability results  for the solution with respect to the initial datum. Finally, we discuss in Section 4 some extensions, regularity results and   
 interesting open problems are pointed out. 

\subsection{ Existence of weak solutions}
\label{EXISTENCE}

Hereupon, we assume
\begin{equation}\label{con1}
   b \in L^1_\loc(U_T), \quad \dive \, b \in L^1((0,T); L^\infty(\R^d)).
\end{equation}   
Actually, it was point out by an anonymous referee that, 
when the drift $b$ satisfies condition \eqref{LPSC}
the derivative and thus the Jacobian of the flow are regularized 
by the noise, and in particular almost bounded, without any hypothesis 
on the divergence. Therefore, it could be that, one does not need 
an $L^\infty$ spatial bound on the divergence of $b$, but just one weaker.
We leave this interesting question open. 
Also, we consider that $u_0 \in L^\infty(\R^d)$. 

\medskip
The next definition tell us in which sense a stochastic process is a weak solution  
of \eqref{trasport}. Hereafter the usual summation convention is used.

\begin{definition}\label{defisolu}  A stochastic process
$u\in L^{\infty}(U_T \times \Omega)$ is called 
a weak $L^{\infty}-$solution of the Cauchy problem \eqref{trasport},
when for any $\varphi \in C_c^{\infty}(\R^d)$, the real value process $\int  u(t,
  x)\varphi(x)
  dx$ has a continuous modification which is a
$\mathcal{F}_{t}$-semimartingale, and for all $t \in [0,T]$, we have $\mathbb{P}$-almost sure

\begin{equation} \label{DISTINTSTR}
\begin{aligned}
    \int_{\R^d} u(t,x) \varphi(x) dx &= \int_{\R^d} u_{0}(x) \varphi(x) \ dx
   +\int_{0}^{t} \!\! \int_{\R^d} u(s,x) \ b^i(s,x) \partial_{i} \varphi(x) \ dx ds
\\[5pt]
    &\;  + \int_{0}^{t} \!\! \int_{\R^d} u(s,x) \ \partial_{i} \varphi(x) \ dx \, {\circ}{dB^i_s}.
\end{aligned}
\end{equation}
\end{definition}

\begin{lemma}\label{lemmaexis1}  Under condition \eqref{con1}, 
there exists a weak $L^{\infty}-$solution $u$ of the Cauchy problem
\eqref{trasport}.
\end{lemma}

\begin{proof}  1. First, let us consider  the following auxiliary Cauchy problem for
the continuity equation, that is to say
\begin{equation}\label{Auxilia1}
 \left \{
\begin{aligned}
&\partial_t v(t,x) + \dive \Big( v(t,x) \; b(t,x+B_{t}) \Big)= 0, 
\\[5pt]
& v(0,x)= u_{0}(x) .
\end{aligned}
\right.
\end{equation}
According to a minor modification of the arguments in DiPerna, Lions \cite{DL}, see
Proposition II.1 (taking only test functions defined on $\mathbb{R}^d$), it follows that,
there exists a function $v \in L^{\infty}(U_T \times \Omega)$, which is a solution  of the auxiliary problem
\eqref{Auxilia1} in the sense that, it satisfies for each test function $\varphi \in C^\infty_c(\R^d)$

\begin{equation}
\label{Auxilia2}
    \begin{aligned}
     \int_{\R^d} v(t,x) \varphi(x) dx&= \int_{\R^d} u_{0}(x) \varphi(x) \ dx
\\[5pt]
     &+ \int_{0}^{t} \!\! \int_{\R^d} v(s,x) \, b(s,x+B_{s}) \cdot \nabla \varphi(x) \ dx ds.
    \end{aligned}     
\end{equation}
One observes that,  the process $ \int v(t,x) \varphi(x) dx$  is adapted, 
since it is the weak limit in $L^{2}([0, T] \times \Omega)$ of adapted processes, see  \cite{Pardoux} Chapter III for details.

\bigskip
2. Now, let us define for each $y \in \R^d$,
$$
  F(y):=\int_{R^d} v(t,x) \, \varphi(x+y) \ dx. 
$$ 
Then, applying the It\^o-Wentzell-Kunita Formula, see Theorem 8.3 of \cite{Ku2},
to $F(B_t)$, it follows from \eqref{Auxilia2}
\begin{equation}\label{auxilia3}
\begin{aligned}
\int_{\R^d} v(t,x) \, \varphi(x+B_{t}) \ dx
&= \int_{\R^d} u_{0}(x) \, \varphi(x) \ dx 
\\[5pt]
 &+ \int_{0}^{t} \!\! \int_{\R^d} b(s,x+B_{s}) \cdot
\nabla \varphi(x+B_{s}) v(s,x) \ dx ds
\\[5pt]
  &+ \int_{0}^{t} \!\! \int_{\R^d}
 v(s,x) \; \partial_i \varphi(x+B_{s}) dx \circ dB_{s}^{i},
\end{aligned}
\end{equation}
where we have used that
$$
    \frac{\partial}{\partial y_{i}}
    \varphi(x+y)=\frac{\partial}{\partial x_{i}} \varphi(x+ y ).
$$  

\bigskip
3. Finally, defining $u(t,x):= v(t, x-B_t)$ 
we obtain from equation (\ref{auxilia3}) that, $u(t,x)$ 
is a weak  $L^{\infty}-$solution
of the stochastic Cauchy problem (\ref{trasport}).
\end{proof}

\section{Uniqueness}
\label{UNIQUE}

We prove the uniqueness result in this section, where it will be considered the divergence-free condition, that is  
\begin{equation}
\label{DIVB}
  \dive \,  b= 0
\end{equation}
(understood in the sense of distributions), and 
also the Ladyzhenskaya-Prodi-Serrin condition \eqref{LPSC}. 

\bigskip
As mentioned in the introduction, under condition 
\eqref{LPSC} we have suitable regularity of the stochastic characteristics. Indeed, under the divergence-free condition the
continuity equation turns to transport equation. Therefore, the main feature of the transport equation, which is the transportation 
property, is used by the authors to show uniqueness in a 
completely different way from the renormalization idea (which exploits  
commutators) used in \cite{Falnanas},   \cite{CO4}, \cite{FGP2}  and  \cite{BrisLion}. Then, we have the 
following  

\begin{theorem}\label{uni} Assume conditions \eqref{LPSC}, and \eqref{DIVB}.  
If $u,v \in L^\infty(U_T \times \Omega)$ are two weak  $L^{\infty}-$solutions 
for the Cauchy problem \eqref{trasport}, with the same initial data 
$u_{0}\in L^{\infty}(\mathbb{R}^{d})$, then
for each $t \in [0,T]$, $u(t)= v(t)$ almost everywhere 
in $\R^d \times \Omega$. 
\end{theorem}

\begin{proof} By linearity, it is enough to show that a weak
$L^{\infty}-$solution $u$ with initial condition $u_{0}(x)=0$ vanishes
identically.  Let $\phi_{\varepsilon}, \phi_{\delta}$ be standard symmetric mollifiers. Thus $u_{\varepsilon}(t,\cdot)=u(t,\cdot)\ast \phi_{\varepsilon}$
verifies
\begin{equation} \label{DISTINTSTRAPP}
\begin{aligned}
    \int_{\R^d} u(t,z) \phi_\varepsilon(y-z) dz &= \int_{0}^{t} \!\! \int_{\R^d} u(s,z) \ b^i(s,z) \partial_{i} \phi_\varepsilon(y-z) \ dz ds
\\[5pt]
    &\;  + \int_{0}^{t} \!\! \int_{\R^d} u(s,z) \ \partial_{i} \phi_\varepsilon(y-z) \ dz \, {\circ}{dB^i_s}.
\end{aligned}
\end{equation}
One remarks that, for each $\ve> 0$ the equation for $u_\ve$ is strong in the analytic sense (which is to
say, it does not need test functions).  

\medskip
Now,   we denote by $b^{\delta}$ the standard mollification of $b$ by $\phi_{\delta}$,
and let   $X_t^{\delta }$  be  the associated flow given by  the SDE (\ref{itoass}) replacing  $b$ by 
 $b^{\delta}$.  
Similarly, we consider $Y^\delta_{t}$, which satisfies the backward SDE \eqref{itoassBac}.

\medskip
Since $\dive \,  b^\delta= 0$ (in other words, the Jacobian of the stochastic 
flow is identically one), for each $\varphi \in C^\infty_c(\R^d)$,  it follows that 
\begin{equation}
\label{PUSHFORWARD}
    \int_{\R^d}  (u \ast \phi_{\varepsilon})(X_t^{\delta}) \; \varphi(y) \ dy 
    =\int_{\R^d}  (u\ast \phi_{\varepsilon})(y) \;  \varphi(Y_t^{\delta}) \ dy ,
\end{equation}
for each $t \in [0,T]$. 

Now, we observe that $ v^{\delta}(t,x)=\varphi(Y_t^{\delta}),$ satisfies the transport equation in the classical sense, that is, it satisfies 

\begin{equation}\label{trasportC}
 \left \{
\begin{aligned}
    & dv^{\delta}(t, x;\omega) + b(t,x)  \nabla v^{\delta}(t, x;\omega)(t, x;\omega) dt + \nabla v^{\delta}(t, x;\omega)(t, x;\omega)  \circ  d B_{t}= 0,
    \\[5pt]
    & v^{\delta}|_{t=0}=  \varphi(x),
\end{aligned}
\right .
\end{equation}

On the other hand, recall that $u_\ve$ is strong in analytic sense. Then we may apply It\^o's formula to the product $$(u\ast \phi_{\varepsilon})(y) \;  \varphi(Y_t^{\delta}),$$ 
and obtain that

\begin{equation}
\label{UNIQ10}
  \begin{aligned}
  \int_{\R^d}  (u\ast \phi_{\varepsilon})(y) & \;  \varphi(Y_t^{\delta}) \ dy  
     =-\int_{0}^{t}\!\!  \int_{\R^d}  u_\varepsilon(s,y)  \; b^{\delta}(s,y) \cdot \nabla [\varphi(Y_s^{\delta})]   \ dy  ds 
\\[5pt]
&- \int_{0}^{t}  \!\! \int_{\R^d}  u_\varepsilon(s,y) \;   \partial_{i} [\varphi(Y_s^{\delta})] dy \  \circ dB_{s}^{i} 
\\[5pt]
 &+\int_{0}^{t} \!\! \int_{\R^d} \varphi(Y_s^{\delta})    \int_{\R^d} u(s,z) \; b(s,z) \cdot \nabla \phi_{\varepsilon}(y-z) \ dz dy ds 
\\[5pt]
& +  \int_{0}^{t} \!\! \int_{\R^d} \varphi(Y_s^{\delta})   \int_{\R^d} 
 u(s,z) \; \partial_{i} \phi_{\varepsilon}(y-z) \ dz dy \  \circ dB_{s}^{i}. 
\end{aligned}
\end{equation}
Then, by integration by parts, we may bring all the derivatives on $\varphi(Y^\delta)$,
hence it will be easy to pass to the limit when $\ve \to 0$ (avoiding commutators).
Therefore, from \eqref{PUSHFORWARD}, \eqref{UNIQ10}, we may write
$$
  \begin{aligned}
  \int_{\R^d}  (u\ast \phi_{\varepsilon})(X^\delta_t) & \;  \varphi(x) \ dx  
     =-\int_{0}^{t}\!\!  \int_{\R^d}  u_\varepsilon(s,y)  \; b^{\delta}(s,y) \cdot \nabla[\varphi(Y_s^{\delta})]   \ dy  ds 
\\[5pt]
&- \int_{0}^{t}  \!\! \int_{\R^d}  u_\varepsilon(s,y) \;   \partial_{i}[\varphi(Y_s^{\delta})] dy \  \circ dB_{s}^{i} 
\\[5pt]
 &+\int_{0}^{t} \!\! \int_{\R^d}   \int_{\R^d} u(s,z) \,  \phi_{\varepsilon}(y-z) \; b(s,z) \cdot \nabla[\varphi(Y_s^{\delta})]  \ dz dy ds 
\\[5pt]
& +  \int_{0}^{t} \!\! \int_{\R^d}  \int_{\R^d} 
 u(s,z) \, \phi_{\varepsilon}(y-z) \; \partial_{i}[\varphi(Y_s^{\delta})]   \ dz dy \  \circ dB_{s}^{i},
\end{aligned}
$$
where we have used that, $\phi_\varepsilon$ is symmetric.

\bigskip
Now for $\delta> 0$ fixed, passing to the limit as $\varepsilon$ goes to $0^+$, we obtain from the above equation 
\begin{equation}
\label{UNIQ20}
  \begin{aligned}
  \int_{\R^d}  u(X^\delta_t) & \;  \varphi(x) \ dx  
     =-\int_{0}^{t}\!\!  \int_{\R^d}  u(s,y)  \; b^{\delta}(s,y) \cdot \nabla[\varphi(Y_s^{\delta})]   \ dy  ds 
\\[5pt]
 &+\int_{0}^{t} \!\! \int_{\R^d}  u(s,y) \; b(s,z) \cdot \nabla[\varphi(Y_s^{\delta})]  \ dy ds.
\\[5pt]
\end{aligned}
\end{equation}
At this point, we use important and recent results obtained by Fedrizzi and  Flandoli \cite{Fre1}, 
more precisely, Lemma 3 and Lemma  5 in that paper. And for safeness of the reader, 
let us recall here the last one:

Lemma 5 (Fedrizzi and  Flandoli) For every $p \geq 1$, there exists $C_{d,p,T}> 0$
such that 
$$
    \sup_{t \in [0,T]} \sup_{x \in \R^d} \mathbb{E}[|\nabla Y_t^\delta(x)|^p] \leq C_{d,p,T}, \quad \text{uniformly in $\delta> 0$}.
$$
Then, applying the Dominated Convergence Theorem we can 
pass to the limit in \eqref{UNIQ20} as $\delta$ goes to $0^+$, to conclude that
\begin{equation}
\label{UNIQ30}
\int_{\R^d} u(X_t) \varphi(x) =0
\end{equation}
for each $\varphi \in C_c^\infty(\R^d)$, and $t \in [0,T]$. 

\bigskip
Finally, let $K$ be any compact set in $\R^d$. Then, we have 
$$
  \begin{aligned}
    \!\! \int_{ K}   \mathbb{E}| u(t,x) |  \ dx &= \lim_{\delta \rightarrow 0}   \!\! \int_K  \mathbb{E}| u(t, X_t^{\delta}(Y_t^{\delta}) ) |  \ dx 
\\[5pt]
     &=   \lim_{\delta \rightarrow 0} \;  \mathbb{E}   \!\!  \int_{Y_t^{\delta}( K)}  | u(t, X_t^{\delta}) |  \ dx  
\\[5pt]     
    &= \mathbb{E}  \!\!  \int_{Y_t( K)}   | u(t, X_t) |  \ dx = 0,
\end{aligned}
$$
 where we have used \eqref{UNIQ30} and the regularity of the stochastic flow.  
Consequently, the thesis of our  theorem is proved. 
\end{proof}

\bigskip
We  also have a representation
formula in terms of the initial condition $u_0$ and the (inverse) stochastic flow associated
to SDE  (\ref{itoass}). Then, we have the following

\begin{proposition}\label{repre} Assume conditions \eqref{LPSC}, and \eqref{DIVB}.  
Given $u_{0}\in L^{\infty}(\mathbb{R}^{d})$, the stochastic process $u(t, x):= u_0(X_{t}^{-1}(x))$ is the 
unique weak $L^{\infty}-$solution of the Cauchy problem \eqref{trasport}.
\end{proposition}

\begin{proof} 1. First, let us assume that $b$ is regular, and  
we denote by $u_0^{\delta}$ the standard mollification of $u_0$.
It is well known, see for instance \cite{Ku3},  that $ u^\delta(t,x):= u_0^{\delta}(X_{t}^{-1}(x))$ is the unique classical
solution of the associated transport equation, thus a  weak $L^{\infty}-$solution o of  \eqref{trasport}
with $u^\delta$ and $u_0^{\delta}$ in place of u and $u_0$.  For each test 
function $\vp \in C_c^\infty(\R^d)$, it follows that 
$$
    \int_{\R^d} u^{\delta}(t,x) \, \varphi(x) \ dx =\int_{\R^d} u_0^{\delta}(X_{t}^{-1}) \, \varphi(x) \ dx 
$$
converges strongly in $L^2([0,T] \times \Omega)$ to 
$$
   \int_{\R^d} u_0(X_{t}^{-1}) \, \varphi(x) \ dx =\int_{\R^d} u(t,x) \, \varphi(x) \ dx .
$$
Then, $u_0(X_{t}^{-1})$ is a weak $L^{\infty}-$solution of the Cauchy problem \eqref{trasport}. 
By Theorem \ref{uni}, uniqueness theorem, it is the only one.

\bigskip
2.  Now, we denote by $b^{\delta}$ the standard mollification of $b$, and let $X_t^{\delta }$ be the associated flow given by the SDE (\ref{itoass}), i.e. replacing  $b$ by 
 $b^{\delta}$. From item 1, we have that $ u^{\delta}(t,x)= u_0(X_{t}^{\delta ,-1})$ is the unique weak $L^\infty$-solution of  \eqref{trasport}
with $u^{\delta}$ and $b^{\delta}$ in place of u and b. Then, applying 
Lemma 3 of \cite{Fre1}, we have for each test function $\vp \in C_c^\infty(\R^d)$ that 
$$
    \int_{\R^d} u^{\delta}(t,x) \, \varphi(x) \ dx =\int_{\R^d} u_0(X_{t}^{\delta ,-1}) \, \varphi(x) \ dx 
$$
converges strongly in  $ \in L^2([0,T] \times \Omega)$ to 
$$
   \int_{\R^d} u_0(X_{t}^{-1}) \, \varphi(x) \ dx = \int_{\R^d} u(t,x) \, \varphi(x) \ dx .
$$
Therefore, $u_0(X_{t}^{-1})$ is the representative formula, which is the unique 
weak $L^{\infty}-$solution of the Cauchy problem \eqref{trasport}. 
\end{proof} 


\section{Stability}
\label{STABILITY}

To end up the well-posedness for the Cauchy problem 
\eqref{trasport}, it remains to show the stability 
 property for the solution with respect to the initial datum. First, we
 establish a weak-stability result, then we show a strong one, 
 assuming the strong convergence of the initial data. 

\begin{theorem}\label{estaweak} Assume conditions \eqref{LPSC}, and \eqref{DIVB}.  
Let $\{u_0^n\}$ be any sequence, with $u_0^n \in L^\infty(\R^d)$ $(n \geq 1)$, converging weakly-star to 
$u_0 \in L^\infty(\R^d)$. Let $u(t,x)$, $u^n(t,x)$ be the unique weak $L^{\infty}-$solution of the Cauchy problem \eqref{trasport},
for respectively the initial data $u_0$ and $u_0^{n}$. Then, for all 
 $t \in  [0, T ]$, and for each function $\vp \in C_c^0(\R^d)$ $\mathbb{P}-$ a.s.
$$
   \int_{\R^d} u^n(t,x) \, \vp(x) \ dx \quad \text{converges to} \quad  \int_{\R^d} u(t,x) \, \vp(x) \ dx \quad \text{$\mathbb{P}-$ a.s.}.
$$
Moreover, if $u_0^{n}$ converges to $u_0$ in $ L^{\infty}(\mathbb{R}^{d})$, then 
\[
 u^{n}(t,x) \ converge  \ to \ u(t,x) \quad \text{in  $L^{\infty}(U_T \times \Omega)$}.
\] 
\end{theorem}

\begin{proof} 1. From Proposition \ref{repre}, we may write
$$
     u^{n}(t,x)=u_0^{n}(X_{t}^{-1}), \quad \text{and} \quad u(t,x)=u_0(X_{t}^{-1}).
$$
Since $\dive \,  b= 0$ (in other words, the Jacobian of the stochastic flow is identically one), for each $\varphi \in C^\infty_c(\R^d)$,  it follows that

\begin{equation}\label{e1}
\int_{\R^d} u^{n}(t,x) \varphi(x) \ dx =\int_{\R^d} u_0^{n}(X_{t}^{-1}) \varphi(x) \ dx =\int_{\R^d} u_0^{n}(x) \varphi(X_{t}) \ dx 
\end{equation}

and analogously

\begin{equation}\label{e2}
\int_{\R^d} u(t,x) \varphi(x) \ dx =\int_{\R^d} u_0(X_{t}^{-1}) \varphi(x) \ dx = \int_{\R^d} u_0(x) \varphi(X_{t}) \ dx .
\end{equation}

Now, by hypothesis $u_0^{n}$  converges  weak$\ast$ $ L^{\infty}(\mathbb{R}^{d})$ to $u_0$, thus 

\[
\int u_0^{n}(x) \varphi(X_{t}) \ dx 
\quad \text{converge to} \quad
\int u_0(x) \varphi(X_{t}) \ dx,
\]
for all  $t \in  [0, T ]$ and
 $\mathbb{P}-$ a.s. From equations  (\ref{e1}) and (\ref{e2}) we deduce that $ u^{n}(t,x)$ converge to $u(t,x)$  weak$\ast$ $ L^{\infty}(\mathbb{R}^{d})$ 
 for all $t \in  [0, T ]$ and  $\mathbb{P}-$ a.s., which finish the proof of weak-stability. 

\medskip
2. Now, let us consider the strong-stability. Again, from Proposition  \ref{repre} we have
\[
u^{n}(t,x)=u_0^{n}(X_{t}^{-1}) \quad \text{and} \quad u(t,x)=u_0(X_{t}^{-1}).
\]
Therefore, we have
$$
   \begin{aligned}
   \sup_{U_T \times \Omega} |u^{n}(t,x)- u(t,x)|&= \sup_{U_T \times \Omega}|u_0^{n}(X_{t}^{-1})- u_0(X_{t}^{-1})|
   \\[5pt]
   &\leq \sup_{\mathbb{R}^{d}}|u_0^{n}(x)- u_0(x)|.
   \end{aligned}
$$
Then, the thesis follows by the 
hypothesis that $u_0^{n}$ converges to $u_0$  in $ L^{\infty}(\mathbb{R}^{d})$.
\end{proof} 


\begin{remark}
It remains open the stability result with respect to $b$ under the Ladyzhenskaya-Prodi-Serrin condition \eqref{LPSC}. 
If additionally,  we assume that $b(t)$ belongs to  $W_{loc}^{1,1}$ or $BV_{loc}$, we can show the  stability 
 property for the solution with  respect to $b$. In fact, the notion 
of renormalized solutions is valid and can  be extended quite easily to a stochastic
framework. For interesting remarks on  renormalized solutions  in the stochastic case see \cite{Falnanas}. 
\end{remark}

\section{Final comments} 

1.  Following the same arguments in the proof of Theorem \ref{uni}, we may also 
treat the case in which b is H\"{o}lder not necessarily bounded, with $\dive b= 0$.
Indeed, from Theorem 7 in \cite{FGP}, if  $b$ is H\"{o}lder not necessarily bounded, then
$X_{t}$ is a H\"{o}lder continuous stochastic flow of diffeomorphisms. Then, 
a similar result of Theorem  \ref{uni} and Proposition \ref{repre} hold.

\bigskip
2. As pointed out by Colombini, Luo and Rauch in \cite{Colom}, 
there exists an important example of $b \in L^\infty \cap W^{1,p}, (\forall p < \infty)$, such that
the  propagation of the continuity in the deterministic transport equation is missing. That is to say,
even if the uniqueness is established in this case, the persistence condition is not, one may start 
with a continuous initial data, but the deterministic solution of the transport equation is not continuous.
However, in the stochastic case we have the persistence property. In fact, let 
$u(t,x)$ be the unique weak $L^{\infty}-$ solution of the Cauchy problem \eqref{trasport},
with $u_0 \in  C_b(\mathbb{R}^{d})$ (i.e. a continuous bounded function). By Proposition \ref{repre}, we have 
$$
   u(t,x)=u_0(X_{t}^{-1}),
$$
and recall that under condition \eqref{LPSC},  $X_{t}$ is a H\"{o}lder continuous stochastic flow of homeomorphisms
(see Section 5 of \cite{Fre2}). Therefore, the continuity of $u$ follows. Analogously, we have the persistence property
when $b$ is H\"{o}lder, not necessarily bounded, with $\dive b= 0$. Indeed, see item 1 above.

\medskip
Also concerning the persistence property, we recall from \cite{Fre1} that a certain Sobolev  regularity 
is maintained under the Ladyzhenskaya-Prodi-Serrin condition, that is, 
$$
   u_0 \in \bigcap_{r\geq 1} W^{1,r}  \Rightarrow u(t,.) \in \bigcap_{r\geq 1} W_{\loc}^{1,r}.
$$

\bigskip
3. It seems to us a very interesting question if  \eqref{LPSC} is sharp, which is to say
if we can consider 
\begin{equation}
\label{LPSC2}
   \frac{d}{p} + \frac{2}{q} \leq 1 \quad \text{instead of} \quad  \frac{d}{p} + \frac{2}{q} < 1.
\end{equation}
This question is posed in particular for SDEs, and personal communications from Krylov and R\"{o}ckner
tell us that it remains open. 

\medskip
If we assume that $b$ does not depend on $t$, $\dive b= 0$
(i.e. autonomous divergence free vector fields), consider dimension $d= 2$, and suppose that \eqref{LPSC2} is true, 
then we may have a uniqueness result for the stochastic transport equation  without any geometrical condition as required in the deterministic case, see in 
particular Hauray \cite{MH}, and also Alberti, Bianchini, Crippa \cite{ABC}. Moreover, these results could open new ideas to solve the Muskat Problem (at least in dimension 2), 
see Chemetov, Neves \cite{CN1,CN2}.

\section*{Acknowledgements}

We would like to thank the anonymous referee 
for providing very constructive comments and help us 
improve this paper.

\smallskip
Wladimir Neves is partially supported by
CNPq through the following grants 
484529/2013-7, 308652/2013-4 
and also by FAPESP through the grant 2013/15795-9. 
Christian Olivera is partially supported by  FAEPEX 1324/12, 
FAPESP 2012/18739-0, FAPESP 2012/18780-0,  and CNPq through the grant 460713/2014-0..  


\end{document}